\documentclass[11pt]{article}

\usepackage{nicefrac}
\usepackage{refcount}
\usepackage{amsfonts}
\usepackage{hyperref}
\usepackage{epsfig}
\usepackage{graphicx}
\usepackage{amsmath}
\usepackage{amssymb}
\usepackage{fullpage}

\newtheorem{theorem}{Theorem}[section]
\newtheorem{definition}[theorem]{Definition}
\newtheorem{proposition}[theorem]{Proposition}
\newtheorem{lemma}[theorem]{Lemma}
\newtheorem{claim}[theorem]{Claim}

\newtheorem{remark}[theorem]{Remark}

\newtheorem{observation}[theorem]{Observation}

\renewcommand{\Re}{\mathbb{R}}
\renewcommand{\L}{\mathcal{L}}

\newenvironment{proof}[1][Proof]{ \noindent \textbf{#1: }}{$\Box$
	\bigskip}

\renewcommand{\Re}{\mathbb{R}}
\renewcommand{\L}{\mathcal{L}}
\newcommand{\F}{\mathcal{F}}
\newcommand{\N}{\mathbb{N}}
\newcommand{\C}{\mathcal{C}}
\newcommand{\G}{\mathcal{G}}
\newcommand{\A}{\mathcal{A}}

\title{An $(\aleph_0,k+2)$-Theorem for $k$-Transversals\footnote{A preliminary version of the paper, which contains a weaker version of the results, was published at the SoCG 2022 conference~\cite{KP22}.}}
\author{Chaya Keller\thanks{Department of Computer Science, Ariel University, Israel. \texttt{chayak@ariel.ac.il}. Research partially supported by the Israel Science Foundation (grant no. 1065/20).}
	\mbox{ }
	and Micha A. Perles\thanks{Einstein Institute of Mathematics, Hebrew University, Jerusalem, Israel.
		\texttt{perles@math.huji.ac.il}}
}

\begin{document}

\maketitle

\begin{abstract}
	A family $\mathcal{F}$ of sets satisfies the $(p,q)$-property if among every $p$ members of $\mathcal{F}$, 
	some $q$ can be pierced by a single point. The celebrated $(p,q)$-theorem of Alon and Kleitman asserts 
	that for any $p \geq q \geq d+1$, any family $\mathcal{F}$ of compact convex sets in $\mathbb{R}^d$ that 
	satisfies the $(p,q)$-property can be pierced by a finite number $c(p,q,d)$ of points. A similar
	theorem with respect to piercing by $(d-1)$-dimensional flats, called $(d-1)$-transversals, was
	obtained by Alon and Kalai.
	
	In this paper we prove the following result, which can be viewed as an $(\aleph_0,k+2)$-theorem
	with respect to $k$-transversals: Let $\mathcal{F}$ be an infinite family of closed balls in $\mathbb{R}^d$, and let
	$0 \leq k < d$. If among every $\aleph_0$ elements of $\mathcal{F}$, some $k+2$ can be pierced by a 
	$k$-dimensional flat, then $\mathcal{F}$ can be pierced by a finite number of $k$-dimensional flats. We derive this result as a corollary of a more general result which proves the same assertion for families of not necessarily convex objects called \emph{near-balls}, to be defined below.

	
	This is the first $(p,q)$-theorem in which the assumption is weakened to an $(\infty,\cdot)$
	assumption. Our proofs combine geometric and topological tools.
\end{abstract}

\section{Introduction}

\subsection{Background}

\noindent \textbf{Helly's theorem and the $(p,q)$-theorem.} The classical Helly's theorem~\cite{Helly23} asserts that if $\mathcal{F}$ is a family of compact convex sets in $\mathbb{R}^d$ and every $d+1$ (or fewer) members of $\mathcal{F}$ have a non-empty intersection, then the whole family $\mathcal{F}$ has a non-empty intersection. 

For a pair of positive integers $p \geq q$, a family $\mathcal{F}$ of sets in $\Re^d$ is said to satisfy the $(p,q)$-property if $|\mathcal{F}|\ge p$, none of the sets in $\mathcal{F}$ is empty, and among every $p$ sets of $\mathcal{F}$, some $q$ have a non-empty intersection, or equivalently, can be pierced by a single point.  A set $P \subset \mathbb{R}^d$ is called a {\em transversal} for $\mathcal{F}$ if it has a non-empty intersection with every member of $\mathcal{F}$, or equivalently, if every member of $\mathcal{F}$ is pierced by an element of $P$. In this language, Helly's theorem states that any family of compact convex sets in $\mathbb{R}^d$ that satisfies the $(d+1,d+1)$-property, has a singleton transversal.

One of the best-known generalizations of Helly's theorem is the $(p,q)$-theorem of Alon and Kleitman (1992), which resolved a 35-year old conjecture of Hadwiger and Debrunner~\cite{HD57}.
\begin{theorem}[the $(p,q)$-theorem \cite{AK}]\label{Thm:pq}
	For any triple of positive integers $p \geq q \geq d+1$ there exists $c=c(p,q,d)$ such that if $\mathcal{F}$ is a family of compact convex sets in $\mathbb{R}^d$ that satisfies the $(p,q)$-property, then there exists a transversal for $\mathcal{F}$ of size at most $c$.
\end{theorem}
In the 30 years since the publication of the $(p,q)$-theorem, numerous variations, generalizations and applications of it were obtained (see, e.g.~\cite{ECK03,HW17,KS21,KST17}). We outline below three variations to which our results are closely related. 

\medskip
\noindent \textbf{$(p,q)$-theorems for $k$-transversals.} The question whether Helly's theorem can be generalized to \emph{$k$-transversals} -- namely, to piercing by $k$-dimensional flats (i.e., $k$-dimensional affine subspaces of $\mathbb{R}^d$) -- goes back to Vincensini~\cite{Vin35}, and was studied extensively. Santal\'{o}~\cite{San40} observed that there is no Helly-type theorem for general families of convex sets, even with respect to $1$-transversals in the plane. Subsequently, numerous works showed that Helly-type theorems for $1$-transversals and for $(d-1)$-transversals in $\mathbb{R}^d$ can be obtained under additional assumptions on the sets of the family (see~\cite{HW17} and the references therein). A few of these results were generalized to $k$-transversals for all $1 \leq k \leq d-1$ (see~\cite{ABM07,AGP02}). 

Concerning $(p,q)$-theorems, the situation is cardinally different. In~\cite{AK95}, Alon and Kalai obtained a $(p,q)$-theorem for \emph{hyperplane transversals} (that is, for $(d-1)$-transversals in $\mathbb{R}^d$). The formulation of the theorem involves a natural generalization of the $(p,q)$-property: 

For a family $\mathcal{G}$ of objects (e.g., the family of all hyperplanes in $\mathbb{R}^d$), a family $\mathcal{F}$ is said to satisfy the $(p,q)$-property with respect to $\mathcal{G}$ if among every $p$ members of $\mathcal{F}$, some $q$ can be pierced by an element of $\mathcal{G}$. A set $P \subset \mathcal{G}$ is called a {\em transversal} for $\mathcal{F}$ with respect to $\mathcal{G}$ if every member of $\mathcal{F}$ is pierced by an element of $P$.
\begin{theorem}[\cite{AK95}]\label{Thm:Alon-Kalai}
	For any triple of positive integers $p \geq q \geq d+1$ there exists $c=c(p,q,d)$ such that if $\mathcal{F}$ is a family of compact convex sets in $\mathbb{R}^d$ that satisfies the $(p,q)$-property with respect to piercing by hyperplanes, then there exists a hyperplane transversal for $\mathcal{F}$ of size at most $c$.
\end{theorem}
As an open problem at the end of their paper, Alon and Kalai~\cite{AK95} asked whether a similar result can be obtained for $k$-transversals, for $1 \leq k \leq d-2$. The question was answered on the negative by Alon, Kalai, Matou\v{s}ek and Meshulam~\cite{AKMM02}, who showed by an explicit example that no such $(p,q)$-theorem exists for line transversals in $\mathbb{R}^3$. Their example can be leveraged by a lifting technique due to Bukh, presented in~\cite{CGH22}, to show that the assertion does not hold for piercing by $k$-flats, for any $1 \leq k \leq d-2$.

\medskip
\noindent \textbf{$(p,q)$-theorems without convexity.} Numerous works obtained variants of the $(p,q)$-theorem in which the convexity assumption on the sets is replaced by a different (usually, topological) assumption. Most of the results in this direction base upon a result of Alon et al.~\cite{AKMM02}, who showed that a $(p,q)$-theorem can be obtained whenever a \emph{fractional Helly theorem} can be obtained, even without a convexity assumption on the elements of $\mathcal{F}$. In particular, the authors of~\cite{AKMM02} obtained a $(p,q)$-theorem for finite families of sets which are a {\em good cover}, meaning that the intersection of every subfamily is either empty or contractible. Matou\v{s}ek~\cite{MAT04} showed that bounded VC-dimension implies a $(p,q)$-theorem, and Pinchasi~\cite{Pinchasi15} proved a $(p,q)$-theorem for geometric hypergraphs whose ground set has a small union complexity. 
Recently, several more general  $(p,q)$-theorems were obtained for families with a bounded Radon number, by Moran and Yehudayoff~\cite{MY20}, Holmsen and Lee~\cite{HL21}, and Pat\'{a}kov\'{a}~\cite{Patakova20}.

\medskip
\noindent \textbf{$(p,q)$-theorems for infinite set families.} While most of the works on $(p,q)$-theorems concentrated on finite families of sets, several papers studied $(p,q)$-theorems for infinite set families.

It is well-known that Helly's theorem for infinite families holds under the weaker assumption that all sets are convex and closed, and at least one of them is bounded. In 1990, Erd\H{o}s asked whether a $(4,3)$-theorem can be obtained in this weaker setting as well. His conjecture -- which was first published in~\cite{BS91} -- was that a $(4,3)$-theorem holds for infinite families of convex closed sets in the plane in which at least one of the sets is bounded.

Following Erd\H{o}s and Gr\"{u}nbaum, who refuted Erd\H{o}s' conjecture and replaced it by a weaker conjecture of his own, several papers studied versions of the $(p,q)$-theorem for infinite families (see~\cite{MMRS15,Muller13}). These papers aimed at replacing the compactness assumption (which can be removed completely for finite families) by a weaker assumption. 

\subsection{Our contributions}

In this paper we study variants of the $(p,q)$-theorem for infinite families $\mathcal{F}$ of sets in $\mathbb{R}^d$. Our basic question is whether the assumption of the theorem can be replaced by the following weaker infinitary assumption, which we naturally call \emph{an $(\aleph_0,q)$-property}: Among every $\aleph_0$ elements of $\mathcal{F}$, there exist some $q$ that can be pierced by a single point (or, more generally, by an element of some family $\mathcal{G}$). We show that despite the apparently weaker condition, $(\aleph_0,q)$-theorems can be obtained in several settings of interest.

The families $\F$ for which we prove our $(\aleph_0,q)$-theorem, consist of not necessarily convex `generalized closed balls' (with no restriction on the radii), called \emph{near-balls},
which we now define. The radius of a ball $D$ is denoted by $r(D)$. For every compact set $B \subset \Re^d$, we denote by $\mathrm{inscr}(B)$ a largest closed ball that is inscribed in $B$, and denote the center of $\mathrm{inscr}(B)$ by $x_B$. We denote by $\mathrm{escribed}(B)$ the smallest closed ball centered at $x_B$ that contains $B$. If there are several balls inscribed in $B$ whose radius is the maximum, we choose $\mathrm{inscr}(B)$ such that $r(\mathrm{escribed}(B))$ is minimal. If there are still several candidates, we choose arbitrarily one of them.

\begin{definition}\label{def:near-balls}
	A family $\F$ of non-empty sets in $\Re^d$ is called \emph{a family of near-balls} if there exists a universal constant $K=K(\F)$ s.t.~for every $B \in \F$, we have:
	\begin{align*}\label{Eq:Near-Balls}
		\begin{split}
			\mbox{\textbf{(1)} } &r(\mathrm{escribed}(B)) \leq K \cdot r(\mathrm{inscr}(B)), \\
			\mbox{\textbf{(2)} } &r(\mathrm{escribed}(B)) \leq K+r(\mathrm{inscr}(B)).
		\end{split}
	\end{align*}
\end{definition}

\noindent Note that any family of near-balls, is in particular, a family of well-rounded sets, as were defined in \cite{Bar21}. For such families in $\Re^3$, it was proved in \cite{Bar21} that the $(2,2)$-property implies an $32K^2$-sized line transversal.
Note also that every family of closed balls, is a family of near-balls.

Our main result concerns $(p,q)$-theorems with respect to $k$-transversals. In this setting, the construction presented in~\cite[Sec.~9]{AKMM02} suggests that no $(\aleph_0,k+2)$-theorem with respect to $k$-transversals can be obtained for general families of convex sets in $\mathbb{R}^d$ where $k<d-1$, since even the stronger $(d+1,d+1)$-property does not imply a bounded-sized $k$-transversal. However, we show that if the convexity assumption is replaced by an assumption that the elements of the family are compact near-balls, then an $(\aleph_0,k+2)$-theorem can be obtained.

\begin{theorem}\label{Thm:Main1}
	For any $d$ and any $0 \leq k \leq d-1$, if an infinite family $\F$ of compact near-balls in $\Re^d$ satisfies the $(\aleph_0,k+2)$-property w.r.t. $k$-flats, then it can be pierced by finitely many $k$-flats.
\end{theorem}
\bigskip

Note that unlike the standard $(p,q)$-theorems, already in the case $k=0$, there is no universal constant $c=c(d)$ such that every family of compact near-balls in $\mathbb{R}^d$ without $\aleph_0$ pairwise disjoint members, can be pierced by at most $c$ points. Indeed, for any $m \in \mathbb{N}$, if the family consists of $\aleph_0$ copies of $m$ pairwise disjoint closed balls, then it satisfies the $(\aleph_0,2)$-property (and actually, even the much stronger $(\aleph_0,\aleph_0)$-property), yet it clearly cannot be pierced by fewer than $m$ points.

Theorem~\ref{Thm:Main1} is sharp, in the sense that both conditions of Definition~\ref{def:near-balls} are needed. To see that the first condition is needed, let $\F=\{e_n\}_{n=1}^{\infty}$ be a family of pairwise intersecting segments inside the unit disc in $\Re^2$ such that no three of them have a point in common. Clearly, $\F$ satisfies the second condition of Definition~\ref{def:near-balls} and the $(\aleph_0,2)$-property w.r.t. points, and yet, it cannot be pierced by a finite number of points.

To see that the less standard second condition is needed, consider the family $\F'=\{B((2^n,0),\frac{2^n}{4}) \cup e_n\}_{n=1}^{\infty}$, where $\{e_n\}$ are the elements of the family $\F$ defined above. (Namely, each set in $\F'$ is the union of a segment near the origin, and a large disc which lies `far' from the origin.) The family $\F'$ satisfies the first condition of  Definition~\ref{def:near-balls} and the $(\aleph_0,2)$-property w.r.t. points, and yet, it cannot be pierced by a finite number of points. This example can be easily modified to make the elements of $\F$ connected. It will be interesting to find out whether such an example can be obtained also under a convexity assumption.

\bigskip

 Theorem \ref{Thm:Main1} requires a significant assumption -- namely, that the elements of the family are \textbf{compact} near-balls. In the case $k=0$ this is clear that the compactness assumption is neccessary. (Consider, for example, the family $ \{  (0,\frac{1}{n}) \}_{n=1,2,\ldots}  \subset \Re  $).
We show by an explicit construction that the compactness assumption is essential also in the case $k=1$, even for planar discs.
\begin{proposition}\label{thm:ex}
	There exists an infinite family $\mathcal{F}$ of open discs in the plane that satisfies the $(3,3)$-property (and so, also the $(\aleph_0,3)$-property) with respect to $1$-transversals (i.e., piercing by lines), but cannot be pierced by a finite number of lines.
\end{proposition} 
Note that such a strong example could not be obtained for families of closed discs in the plane, since, by Theorem~\ref{Thm:Alon-Kalai}, a family of compact convex sets in the plane that satisfies the $(3,3)$-property with respect to piercing by lines, admits a bounded-sized line transversal.

\medskip A simple consequence of Theorem~\ref{Thm:Main1} is the following cute statement:
\begin{proposition}\label{Prop:Simple}
	Let $\F$ be an infinite family of balls in $\mathbb{R}^d$. If for any $p \in \mathbb{N}$, $\F$ contains $p$ pairwise disjoint balls, then $\F$ contains an infinite sequence of pairwise disjoint balls.
\end{proposition} 
This statement can be proved without using the full strength of Theorem~\ref{Thm:Main1}, and yet, is not completely trivial. It should be noted that Proposition~\ref{Prop:Simple} does not hold for families of general convex sets. We leave the riddle of finding a counterexample to the reader.

\subsection{Follow-up work}

After the initial version of this paper, which contained a proof of Theorem~\ref{Thm:Main1} only in the special case of families of $(r,R)$-fat balls, had appeared at the SoCG'22 conference~\cite{KP22}, several works followed up on its results.

Ghosh and Nandi~\cite{GN22} obtained several generalizations of the results of~\cite{KP22}. First, they obtained \emph{colorful} versions of the results of~\cite{KP22}, in the spirit of B\'{a}r\'{a}ny's colorful Helly theorem~\cite{Bar82}. The generalization to a colorful version is more challenging than usual, as for infinite families, it is not even clear what the natural notion of `colorful' is, and indeed, two different notions are studied in~\cite{GN22}. Second, the authors of~\cite{GN22} obtained a generalizarion of the results of~\cite{KP22} to families of bounded convex sets in $\mathbb{R}^2$ (for $k=1$) and of bounded convex balls in $\mathbb{R}^d$ (for all $k<d$), under the additional assumption that the family is a so-called \emph{$k$-growing sequence} (see~\cite[Def.~16, Thms.~17,18]{GN22}). The latter result is a special case of our Theorem~\ref{Thm:Main1}, which obtains the same conclusion for families of balls with no restriction on the diameter and with no `growing sequence' assumption. 

Chakraborty, Ghosh, and Nandi~\cite{CGN23} gave necessary and sufficient conditions for an infinite collection of axis-parallel boxes in
$\mathbb{R}^d$ to be pierceable by finitely many axis-parallel $k$-flats, where $0 \le k < d$. They also discussed colorful versions of their result. 

In a very recent work, Jung and P{\'a}lv{\"o}lgyi~\cite{JP23} suggested an alternative approach toward obtaining $(\aleph_0,q)$-theorems: Instead of proving the infinitary statement directly, one may try to show that the $(\aleph_0,q)$-property implies the existence of a $(p,q)$-property for a sufficiently large $p$, and then deduce the $(\aleph_0,q)$-theorem from a corresponding $(p,q)$-theorem. They proved~\cite[Lem.~17]{JP23} such a reduction statement for families which satisfy a \emph{fractional Helly theorem}, by a complex argument which uses and further develops techniques introduced by Holmsen~\cite{Hol20}. Then, they proved~\cite[Thm.~10]{JP23} a fractional Helly theorem with respect to piercing by $k$-flats for families of so-called \emph{convex $\rho$-fat sets}, which include convex near-balls. Furthermore, they deduced a $(p,k+2)$-theorem w.r.t.~piercing by $k$-flats for convex $\rho$-fat sets (\cite[Thm.~2]{JP23}). This allowed Jung and P{\'a}lv{\"o}lgyi to apply the reduction strategy to prove an  $(\aleph_0,k+2)$-theorem w.r.t.~piercing by $k$-flats for families of convex $\rho$-fat sets, thus providing an alternative proof for Theorem~\ref{Thm:Main1} in this setting (see~\cite{JP23}, Thm.~9 and the paragraphs after it). 

It should be mentioned however that since the fractional Helly theorem of~\cite{JP23} relies crucially on the convexity assumption, the reduction approach of~\cite{JP23} implies the assertion of Theorem~\ref{Thm:Main1} only for families of \emph{convex} near-balls. On the contrary, our approach which tackles the $(\aleph_0,k+2)$-problem directly using geometric tools, allows proving Theorem~\ref{Thm:Main1} for families of general (i.e., not necessarily convex) near-balls. 


\bigskip

\noindent \textbf{Organization of the paper.}
The paper is organized as follows. Definitions and notations that will be used throughout the paper are presented in Section~\ref{sec:def}, and the proof of Theorem~\ref{Thm:Main1} is presented in Section~\ref{sec:main}. The Proof of Proposition~\ref{thm:ex} ispresented in Section \ref{sec:anti}.

\section{Definitions, Notations, and Basic Observations} \label{sec:def}


We use the following classical definitions and notations.
\begin{itemize}
	\item For $0 \leq k \leq d-1$, a $k$-flat in $\mathbb{R}^d$ is a $k$-dimensional affine subspace of $\mathbb{R}^d$ (namely, a translate of a $k$-dimensional linear subspace of $\mathbb{R}^d$). In particular, a $0$-flat is a point, an $1$-flat is a line, and a $(d-1)$-flat is a hyperplane. Note that the translate through the origin of a $k$-flat $J$, is just the difference $J-J$.
	
	\item The \emph{direction} of a $k$-flat ($k>0$) in $\mathbb{R}^d$ is defined as follows. First, the $k$-flat is translated such that it will pass through the origin. Then, its direction is defined as the great $(k-1)$-sphere in which the $k$-flat intersects the sphere $\mathcal{S}^{d-1}$, namely, $(J-J) \cap \mathcal{S}^{d-1}$. (This definition follows~\cite{AGP02}.)
	
	\item A $k_1$-flat and a $k_2$-flat are called \emph{parallel} if the direction of one of them is included in the direction of the other. (Equivalently, this means that if both are translated to pass through the origin, then one translate will include the other. Note that this relation is not transitive, and that two flats of the same dimension are parallel, if and only if one of them is a translate of the other.)
	
	
	\item A family $\F=\{B_{\alpha}  \}_{\alpha}$ of sets in $\Re^d$, is \emph{independent w.r.t. k-flats} (or \emph{k-independent}) if no $k$-flat $\pi \subset \Re^d$ intersects $k+2$ $B_{\alpha}$'s or more, and if $|\F|=n<k+2$, no $(n-2)$-flat meets all members of $\F$. (That is, every selection of $j$ points from $j$ members of $\F$, $1 \leq j \leq k+1$, is an affinely independent set.)
	
	\item The distance between a point $x \in \Re^d$ and a compact set $B \subset \Re^d$ is the minimal distance between $x$ and some point $y \in B$, and is denoted by $dist(B,x)$.
\end{itemize}

In this paper we consider the compactification $\C$ of $\Re^d$, in which we add an ``point at infinity'' to all the rays in the same direction. This means that every line in $\C$ is incident to two points at infinity (and not only to one point at infinity as in the projective space). This compactification $\C$ is homeomorphic to to the closed ball $B(0,\pi/2) \subset \Re^d$. Indeed, we can define a homeomorphism $h$ that transfers all the finite points in $\C$ (which are actually points of $\Re^d$) into the open ball with radius $\pi / 2$ centered at the origin, by defining 
	\[
0_{\Re^d} \mapsto 0_{\Re^d}
\] 
	\[
0\neq x \mapsto \frac{x}{||x||}arctg(||x||),
\] 
and transfer continuously the point at infinity in $\C$ into the boundary of this ball. (Each ray in $\Re^d$ that emanates from the origin is transferred to an initial segment of itself.) 

For every finite point $p \in \C$ (which is a point in $\Re^d$), an (open) neighborhood of $p$ is $B^{\circ}(p,\epsilon) \subset \Re^d$, where $B^{\circ}(p,\epsilon)$ is the open ball of radius $\epsilon$ centered at $p$. 
For the sake of convenience, we define the neighborhoods of a point at infinity in $\C$, only with respect to the finite part of $\C$: Let $p \in \Re^d$ with $||p||=1$. Define a \emph{(finite)-$\epsilon$-neighborhood} of the point at infinity on the ray $\overrightarrow{ 0p}$, to be
	\[
  \{  x \in \Re^d \subset \C: |   \frac{x}{||x||}   -  p   |<\epsilon   \land ||x||>\frac{1}{\epsilon}   \}.
\] 
This (somewhat non-standard) definition takes into account only the finite points in the neighborhood of the point at infinity, which is more reasonable in our context, in which all the elements of $\F$ are in the finite part of $\C$. 
Under these two definitions of neighborhoods, for any sequence $\{x_n\}_{n=1}^{\infty}$ of finite points in $\C$, the limit $\lim_{n \to \infty}x_n$, if it exists, is defined in the standard topological definition, and this limit is a point in $\C$ (that can be either a finite point or a point at infinity).

\bigskip
The following proposition, which is a general application of compactness, is crucial in the sequel.

\begin{proposition}\label{Prop:Compactness}
	Let $\F$ be a family of compact sets in $\mathbb{R}^d$, let $0 \leq k \leq d-1$ and $m \in \mathbb{N}$. If any finite subfamily of $\F$ can be pierced by at most $m$ $k$-flats, then $\F$ can be pierced by at most $m$ $k$-flats.
\end{proposition}

\begin{proof}[Proof of Proposition~\ref{Prop:Compactness}] 
	Any $k$-flat $\pi \subset \Re^d$ can be represented as 
	$$  \pi =\{  c+\lambda_1 v_1+\ldots+  \lambda_k v_k      : \lambda_i \in \Re      \},   $$ 
	where $c$ is the point on $\pi$ closest to the origin, and $\{v_1, \ldots,v_k\}$ is an orthonormal basis of the vector subspace $\pi-\pi = \{x-y: x,y \in \pi\}$ which is parallel to $\pi$. (This actually means that the vector $\vec{0c}$ is orthogonal to each $v_i$.)
	
	Assign to each $k$-flat $\pi$ all the corresponding $(k+1)$-tuples of the type $\{  c,v_1 ,\ldots, v_k\} $ that represent $\pi$. Note that while $c$ is uniquely determined by $\pi$, the orthonormal basis is not. We obtain a representation of all $k$-flats in $\Re^d$ as $(k+1)$-tuples of $d$-vectors
	$$  \mathcal{A}  = \{   (c,v_1 ,\ldots, v_k)  : c,v_i \in \Re^d \land \forall i \neq j, v_i \perp v_j \land  v_i \perp c\land||v_i||=1
	\}   \subset \Re^{d(k+1)} .             $$
	
	By the conditions of the proposition, we can assume w.l.o.g. that there exists a finite subfamily $\F_0 \subset \F$ that cannot be pierced by $m-1$ $k$-flats. Let elements of
	$$  \mathcal{A}^m  = \{   (c^1,v_1^1 ,\ldots, v_k^1, \ldots, c^m,v_1^m ,\ldots, v_k^m)  :  \forall  1\leq j \leq m,  (c^j,v_1^j ,\ldots, v_k^j) \in \mathcal{A}
	\}   \subset \Re^{d(k+1)m}            $$
	(the cartesian product of $m$ copies of $\mathcal{A}$) represent $m$-tuples of $k$-flats in $\Re^d$.
	
	Note that $\mathcal{A}^m$ is not compact (as a subset of $\Re^{d(k+1)m}$),  since $||c^j||$ may be arbitrarily large. However, for any fixed set $B \in \F$, the subset $\Pi_B\subset \mathcal{A}^m $ that contains the representations of all $m$-tuples of $k$-flats intersecting $B\cup \F_0$, is a compact subset of $\mathcal{A}^m$. 
	Indeed, these norms $||c^j||$ of all $m$ components $c^j$, are bounded by the radius of a ball centered at the origin that includes the union $\bigcup\F_0$.
	
	Consider the family $\{\Pi_B\}_{B \in \F}$, where $\Pi_B$ is the set of representions all $m$-tuples of
	$k$-flats that pierce $B \cup \F_0$. (For each such $m$-tuple of $k$-flats, we take all possible $(d(k+1)m)$-tuples that represent it.) Each $\Pi_B$ is compact, and by the
	assumption, any finite subfamily $\{\Pi_{B_i}\}_{i=1}^n$ has non-empty intersection (that contains the representation of some $m$-tuple of $k$-flats, $\bigcap_{i=1}^n\Pi_{B_i}=\Pi_{\cup_{i=1}^n B_i}$, that together pierce $\F_0\cup \{B_1,\ldots,B_n\}$).
	Therefore, by the finite intersection property of compact sets, the whole family $\{\Pi_B\}_{B \in \F}$ has non-empty intersection. Any element in this non-empty intersection represents an $m$-tuple of $k$-flats that pierce together all the sets in $\F$.
\end{proof}

\section{Proof of the Main Theorem} \label{sec:main}

We restate Theorem \ref{Thm:Main1}, in a formulation that will be more convenient for the proof:

\medskip

 \noindent \textbf{Theorem \ref{Thm:Main1} -- restatement.}
Let $\F$ be a family of compact  near-balls in $\mathbb{R}^d$. Let $0 \leq k \leq d-1$. Then exactly one of the two following conditions  holds:
\begin{itemize}
	\item $\F$ can be pierced by a finite number of $k$-flats.
	
	\item $\F$ contains an infinite sequence of sets that are $k$-independent (i.e., no $k$-flat pierces $k+2$ of them).
\end{itemize}

\subsection{The induction basis for $k=0$ and every $d$}

The proof of Theorem \ref{Thm:Main1} is by induction, passing from $(k-1,d-1)$ to $(k,d)$. The induction basis is the following Lemma, which is the
case $k=0$ of Theorem \ref{Thm:Main1}.

\begin{lemma}\label{Thm:1-dim} 
	Let $\F$ be a family of compact near-balls in $\mathbb{R}^d$. Then exactly one of the two following conditions holds:
	\begin{itemize}
		\item $\F$ can be pierced by a finite number of points.
		
		\item $\F$ contains an infinite sequence of pairwise disjoint sets.
	\end{itemize} 
\end{lemma}

Before proceeding into the proof of Lemma \ref{Thm:1-dim}, we prove a reduction to the case where all elements of $\F$ lie within a closed bounded ball $B(0,R) \subset \Re^d$.

\begin{claim}\label{cl:bounded}
	Suppose Lemma \ref{Thm:1-dim} holds when $\bigcup_{B\in \F}B \subset B(0,R)$, for some $R>0$. Then Lemma \ref{Thm:1-dim} holds without any restriction.
\end{claim}	
\begin{proof}
	Consider the set $\{ dist(insc(B_i),0)  \}_{i=1,2,\ldots} \subset \Re^+$ where $B_i\in \F$ for all $i$. If this set is unbounded, then it contains an infinite sequence $\{ dist(insc(B_i),0) : B_i \in \F     \}_{i=1,2,\ldots}$ 
	that tends to infinity. By condition 2 of Definition \ref{def:near-balls}, $\lim_{i \to \infty} dist(B_i,0)=\infty$, and one can inductively choose an infinite subsequence $\{B_{i_n}\}_{n=1}^{\infty}  $ of $\{B_i\}_{i=1}^{\infty}$, such that $dist(B_{i_{n+1}},0)>max \{||x||:x\in B_{i_n}\}$.
	In this case, the second assertion holds, and we are done.
	
	From now on we assume that there exists some $R \in \Re$ satisfying $K(\F)+2=K+2<R $, such that for any $B \in \F$, $dist(insc(B),0) \leq R-K-2$. Replace each $B \in \F$ with $r(insc(B))>1$, by some closed ball $B' \subset insc(B)$ with $r(B')=1$, such that $dist(0,insc(B))=dist(0,B')$. Let $\F'$ be the obtained family. Clearly, $\F'$ is a family of compact near-balls. For any $B\in \F'$, $insc(B)$ is contained in $ B(0,R-K)$, and therefore, by condition 1 of Definition \ref{def:near-balls}, $ B \subset B(0,R)$.
	
	By the assumption of our claim, either $\F'$ can be pierced by finitely many points, or $\F'$ contains an infinite sequence $\F'' \subset \F'$ of pairwise disjoint balls. In the first case, the finite piercing set of $\F'$ pierces $\F$ as well.
	
	In the second case, remove from $\F''$ all balls with radius 1. There are only finitely many such balls, since $\F'' \subset \F' \subset B(0,R)$, and the elements of $\F''$ are pairwise disjoint.  
	After removing from $\F''$ all balls with radius 1, we are left with an infinite subfamily of sets each of which belongs to $\F$ (since in the transition from $\F$ to $\F'$, all the modified sets become radius-1 balls of $\F'$), which are pairwise disjoint.
\end{proof}

\begin{proof}[Proof of Lemma \ref{Thm:1-dim}]
	By Claim \ref{cl:bounded} we can assume that there exists $R>0$ such that each $B \in \F$ is contained in $B(0,R)$.
		Recall that by our convention in Section \ref{sec:def}, each $B \in \F$ contains a ``center'' $x_B \in B$, which is the center of both $insc(B)$ and $escribed(B)$. 
	
	Each $x \in B(0,R) $ is of exactly one of the two following types:
	
	\noindent \textbf{Type (a):} For each $\delta>0$, there exists some $B \in \F$, such that $x_B \in B^{\circ}(x,\delta) $, $r(insc(B))<\delta$ and $x \notin B$. 
	
	\noindent \textbf{Type (b):} There exists $0 < \delta=\delta(x)$ such that for any $B \in \F$ with $x_B \in B^{\circ}(x,\delta)$, the following holds: Either $r(insc(B))\geq \delta$ or $x \in B$. 
	
	\medskip
	If $B(0,R) $ contains some point $x$ of type $(a)$, then there exists an infinite sequence of pairwise disjoint sets in $\F$ (whose centers tend to $\{x\}$). Indeed, start with $\delta_0=1$ and pick some $B_0 \in \F$, with $x_{B_0} \in B^{\circ}(x,\delta_0) $, with $r(insc(B_0))<\delta_0$ and $B_0 \cap \{x\} = \emptyset$. Since $B_0$ is compact, it has a positive distance $\epsilon$ from $x$. Let $\delta_1=\frac{\epsilon}{10K}$ and pick some $B_1 \in \F$, with $x_{B_1} \in B^{\circ}(x,\delta_1) $, $r(insc(B_1))< \delta_1$ and $B_1 \cap \{x\} = \emptyset$. By condition 1 of Definition \ref{def:near-balls}, $r(escribed(B_1))<\frac{\epsilon}{10}$, hence $B_1$ is disjoint from $B_0$, and closer to $x$ than $B_0$.
	Continue in the same manner to construct an infinite sequence of pairwise disjoint sets in $\F$.
	
	The remainig case is where each $x \in B(0,R)$ is of type $(b)$. Then for each $x \in B(0,R)$ there exists $0 < \delta=\delta(x)$ such that any $B \in \F$ with $x_B \in B^{\circ}(x,\delta)$ can be pierced by finitely many points, say, by $f(x)$ points. (Note that the exact value of $f(x)$ depends on the choice of $\delta=\delta(x) $.)

	By the finite intersection property of compact sets in $\Re^d$, the cover $ \bigcup_{ x \in B(0,R)}  B^{\circ}(x,\delta(x))     $ of $B(0,R)$ has a finite sub-cover $B(0,R) \subset  \bigcup_{ i=1}^k  B^{\circ}(x_i,\delta(x_i))     $. Since all the sets in $\F$ whose center lies in $B^{\circ}(x_i,\delta(x_i))$ can be pierced by $f(x_i)$ points, it follows that all the elements of $\F$ can be pierced by at most $\Sigma_{i=1}^k f(x_i)$ points.
\end{proof}

\begin{remark}
	The reader can verify that the assertion of Lemma \ref{Thm:1-dim} holds, even where we replace the first condition of Definition \ref{def:near-balls} by the following weaker condition: $$\lim_{r \to o}    \sup \{    r(escribed(B))  : B \in \F, r(insc(B)) \leq r \}     =0.$$
\end{remark}

\subsection{The induction step}
Let $1 \leq k,d\in \N$, and assume Theorem \ref{Thm:Main1} holds for $k-1,d-1$. Let $\C$ be the compactification of $\Re^d$, described in Section \ref{sec:def}. 
\begin{definition}
	
	\begin{enumerate}		
		\item A point $x \in \C$ is \emph{weak} if there exists a positive $\epsilon=\epsilon(x)$ such that the family $\{  B\in \F : x_B \in B^{\circ} (x,\epsilon)\}$ can be pierced by finitely many $k$-flats.
		\item A point $x \in \C$ is \emph{strong} if for every $\epsilon>0$ the family $\{  B\in \F : x_B \in B^{\circ} (x,\epsilon)\}$ cannot be pierced by finitely many $k$-flats.
	\end{enumerate}
\end{definition}
\noindent Clearly, each $x \in \C$ is either weak or strong.

\medskip

If every $x \in \C$ is weak, then, by compactness, the open cover $\bigcup_{x\in\C}B^{\circ}(x,\epsilon(x))$ of $\C$, has a finite subcover, and therefore, the whole family $\F$ can be pierced by finitely many $k$-flats. Hence, from now on, we assume that $\C$ contains (either a finite or an infinite) strong point, $x_0$.

\begin{claim}
	\label{cl:comp}
	Let $x_0 \in \C$ be a strong point. Then there exists a sequence $\{B_n\}_{n=1}^{\infty} \subset \F$ with ``centers'' $x_n:=x_{B_n}$, such that $\lim_{n \to \infty}x_n=x$, and the family  $\{B_n\}_{n=1}^{\infty}$ cannot be pierced by finitely many $k$-flats.
\end{claim}
\begin{proof}[Proof of Claim \ref{cl:comp}]
	We construct the sequence $\{B_n\}_{n=1,2,3,\ldots} \subset \F$ as follows.	
	For each $m \in \mathbb{N}$, we find in $\{B\in \F: x_B \in B^{\circ}(x,\frac{1}{m})  \}$ a finite family $G_m$ of sets that cannot be pierced by $m$ $k$-flats (this is possible, by Proposition \ref{Prop:Compactness}). We define the sequence $\{B_n\}_{n=1,2,3,\ldots}$ as $\bigcup_{m \in \mathbb{N}} G_m$.
	Namely, we arbitrarily order the sets in each $G_m$ and add them to the sequence, allowing repetitions, starting with $m=1$, proceeding to $m=2$, etc..
	We have $\lim_{n \to \infty} x_n = x_0$, (no matter whether $x_0$ is a finite point or a point at infinity). Furthermore, $\{B_n\}_{n=1,2,3,\ldots}$, cannot be pierced by $m$ $k$-flats (for any $m \in \mathbb{N}$) since it includes the family $G_{m}$ that cannot be pierced by $m$ $k$-flats by its construction. Hence, $\{B_n\}_{n=1,2,3,\ldots}$ cannot be pierced by a finite number of $k$-flats, as asserted.   
\end{proof}

The rest of the proof depends on the strong point $x_0$ being finite or a point at infinity. Though there are some similar parts in the analysis of the cases, there are also essential different arguments, so for the sake of convenience we treat each case separately, and shorten some explanations in the second case.

\bigskip

\noindent \textbf{Case 1: The strong point $x_0$ is a point at infinity:} Assume w.l.o.g. that $x_0$ is the point at infinity on the positive ray of the $d$-axis. Let $\G=\{B_n\}_{n=1}^{\infty}\subset \F$ be the sequence guaranteed by Claim \ref{cl:comp}. We remove from $\G$ all the sets that meet the $d$-axis, clearly the remaining set $\G$ (we still call it $\G$) still has the property of not being pierced by finitely many $k$-flats.
We project each $B_n \in \G$ onto its first $d-1$ coordinates. Let the resulting set be $\G'$. Note that since each ball is projected to a ball of the same radius, $\G'$ is a family of compact near-balls and $K(\G')\leq K(\F)$.

By the induction hypothesis, either $\G'$ can be pierced by a finite number of $(k-1)$-flats in $\mathbb{R}^{d-1}$, or else it includes a sequence of $(d-1)$-dimensional sets that are $(k-1)$-independent.

The first option cannot happen, as otherwise, one could pierce $\G$ by finitely many $k$-flats (which are the pre-images of the $(k-1)$-flats-transversal in $\Re^{d-1}$ under the projection), contrary to the choice of $\G$. 

Hence, there exists a sub-sequence $\bar{\G} = \{B_{n_l}\}_{l=1,2,\ldots} \subset \G$ of sets whose projections are $(k-1)$-independent in $\mathbb{R}^{d-1}$. Clearly, $\lim_{l \to \infty} x_{B_{n_l}} = x_0$.
From now on, we restrict our attention to this sequence, $\bar{\G}$, and construct inductively a subsequence $\{B^i\}_{i=1,2,\ldots} \subset \bar{\G}$ that will be $k$-independent in $\mathbb{R}^{d}$.

The first $k+1$ elements $B^1,\ldots,B^{k+1} \in \bar{\G}$ can be chosen arbitrarily. (They will be $k$-independent in $\Re^d$, since their projections are $(k-1)$-independent in $\Re^{d-1}$.) Assume we already chose $m$ elements of the subsequence. To choose the next set, we look at each of the ${{m} \choose {k+1}}$ $(k+1)$-tuples of sets that have already been chosen. Let such a $(k+1)$-tuple be $B^{i_1},\ldots,B^{i_{k+1}} \in \bar{\G}$. By assumption, the corresponding projections on the first $d-1$ coordinates cannot be pierced by a $(k-1)$-flat in $\mathbb{R}^{d-1}$. This implies that no $k$-flat that is parallel to the $d$-axis can pierce the whole $(k+1)$-tuple,  $B^{i_1},\ldots,B^{i_{k+1}} $.

Consider the family $U$ of all $k$-flats that pierce the whole $(k+1)$-tuple, $B^{i_1},\ldots,B^{i_{k+1}}$.
Let $U_0=\{J-J  :  J \in U\}$ be the family of all translates to the origin of the $k$-flats in $U$. As none of the $k$-flats in $U$ is parallel to the $d$-axis, no element of $U_0$ contains the point $(0,0,\ldots,0,1)=\hat{x} \in \mathcal{S}^{d-1}$. By compactness of the elements of $\F$, this implies that there exists an $\epsilon_0>0$, such that the $\epsilon_0$-neighborhood of $\hat{x}$ in $\mathcal{S}^{d-1}$ is disjoint from all elements of $U_0$.

This argument works for each $(k+1)$-subset of the $m$ elements we have already chosen. Let $\epsilon_1$ be the minimum of the $\epsilon_0$ values. 

\begin{observation} \label{obs:cone}
	The cone $\mathcal{C}(\hat{x},\epsilon_1)=
	\{ x \in \Re^d  :  \langle \hat{x},x \rangle  \geq ||x||\cos \epsilon_1\} \setminus\{0\}$ around the $d$-axis with apex at the origin and aperture $\epsilon_1$, is disjoint from each $k$-flat $J_0$ where $J_0 \in \A $, and $$\A  =  \{ J-J: J \mbox{ is a }k\mbox{-flat that intersects }B^{i_1},\ldots,B^{i_{k+1}} \mbox{ for some }1 \leq i_1<\ldots<i_{k+1} \leq m\}.$$ 
	(Each element in $\A$ is a translate-to-the-origin of some $k$-flat $J$ that meets $k+1$ sets among the $m$ previously-chosen $B^i$'s.)
\end{observation} 

Let $D \in \Re^+$ be a large constant such that $B^1,\ldots,B^m \subset B(0,D)$. We claim that for all but finitely many elements $B_{n_l}$ of $\bar{\G}$, even the inflated sets $B_{n_l}+B(0,D)$ (in each of which every point in the set is surrounded by a $D$-radius ball) lie entirely within $\mathcal{C}(\hat{x},\epsilon_1)$.

\begin{claim} \label{cl:cone}
	There exists $l_1 \in \N$ such that for every $l \geq l_1$, $B_{n_l}+D\cdot B(0,1) \subset \mathcal{C}(\hat{x},\epsilon_1)$.
\end{claim}

\noindent Let us, first, complete the proof of Case 1 assuming Claim \ref{cl:cone}, and then prove the claim.

By claim \ref{cl:cone}, we can choose a sufficiently large $l$, such that $B_{n_l}$ is disjoint from each $k$-flat in $\A$. Let $B^{m+1}=B_{n_l}$ (See Figure \ref{fig:newfig1}.) We claim that $B^1,\ldots,B^{m+1}$ is $k$-independent.

Indeed, assume to the contrary that some $k$-flat $J$ intersects $k+2$ sets from  $B^1,\ldots,B^{m+1}$. $J$ must intersect $B^{m+1}$ and $k+1$ other sets $B^{i_1},\ldots,B^{i_{k+1}} $, for $1 \leq i_1<\ldots<i_{k+1} \leq m$. Note that $J=J_0+q$, where $J_0=J-J$ is the translate of $J$ to the origin, and $q$ is a point in $J \cap B^1$. Therefore, $||q|| \leq D$ and $B^{m+1}\cap (J_0+q)\neq \emptyset$. This means that $(B^{m+1}-q) \cap J_0 \neq \emptyset$. But $B^{m+1}-q$ is contained in the inflation of $B_{n_l}=B^{m+1} \subset B^{m+1}+D\cdot B(0,1)$, which is disjoint from $J_0$ by Observation \ref{obs:cone} and Claim \ref{cl:cone}, a contradiction.

The only part that remains is to prove Claim \ref{cl:cone}.
\begin{figure}[tb]		
	\begin{center}
		\scalebox{0.7}{
			\includegraphics{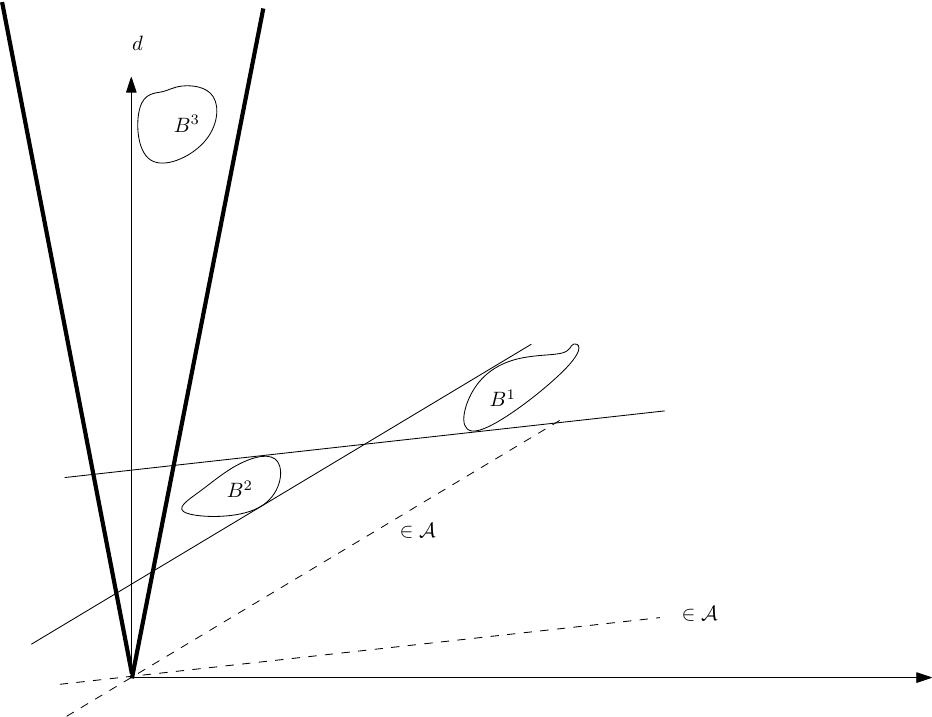}
		}
	\end{center}
	\caption{An illustration for the proof of Case 1 for $k=1, m=d=2$.}
	\label{fig:newfig1}
\end{figure}
\begin{figure}[tb]		
	\begin{center}
		\scalebox{0.7}{
			\includegraphics{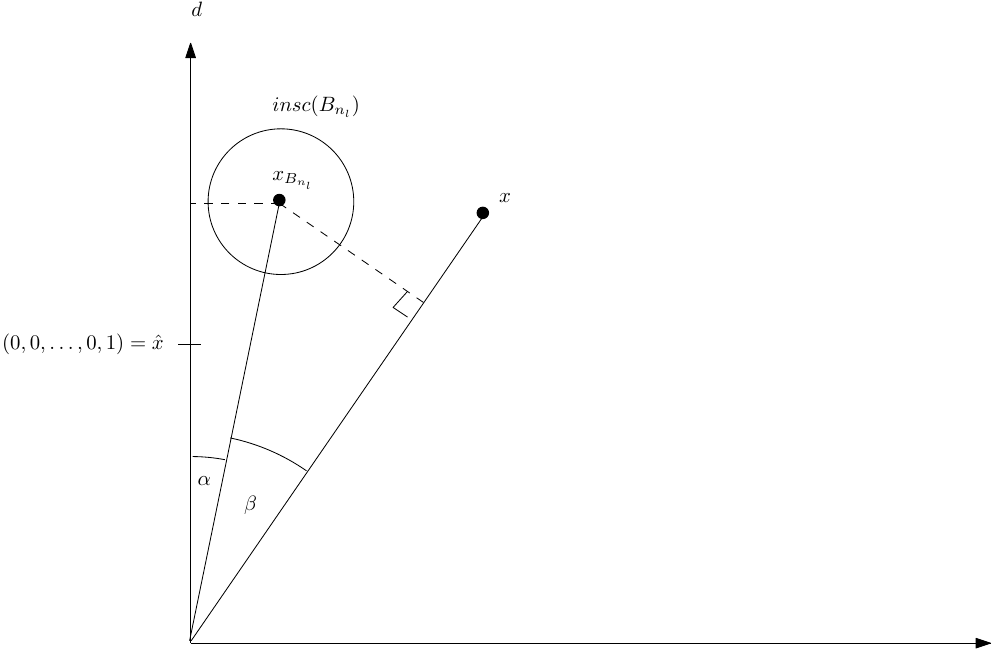}
		}
	\end{center}
	\caption{An illustration for the proof of Claim \ref{cl:cone}.}
	\label{fig:newfig2}
\end{figure}

\begin{proof}[Proof of Claim \ref{cl:cone}]
	Let $K=K(\F)$ and let $\epsilon'=\frac{\epsilon_1}{1+\frac{\pi}{2}(K+D)}$. Since $\lim_{l \to \infty}x_{B_{n_l}}=x_0$, there exists $l_1 \in \N$ such that for every $l>l_1$, $||x_{B_{n_l}}|| \geq \frac{1}{\epsilon'}$ and $x_{B_{n_l}} \in \C(\hat{x},\epsilon')$. Let $\alpha = \measuredangle (\hat{x},0,x_{B_{n_l}})< \epsilon'$, and let $x \in B_{n_l} + D \cdot B(0,1)$. We have to prove that $x \in \C(\hat{x},\epsilon_1)$. Let $\beta=\measuredangle (x_{B_{n_l}},0,x)$ (see Figure \ref{fig:newfig2}). We have to prove that $\measuredangle (\hat{x},0,x )< \epsilon_1$. Since $\measuredangle (\hat{x},0,x )\leq\alpha+\beta$ (note that the 3 rays we consider are not neccessarily on the same plane!), it is sufficient to prove that $\alpha+\beta<\epsilon_1$. Let $r_{n_l}=r(insc(B_{n_l}))$. Since $B_{n_l}$ is not pierced by the $d$-axis (as was ensured at the beginning of this case) $dist(x_{B_{n_l}}, d \mbox{-axis})<r_{n_l}$.
	
	By the first condition of Definition \ref{def:near-balls}, $||x - x_{B_{n_l}}|| \leq K \cdot r_{n_l}+D$. Furthermore, the distance of $x_{B_{n_l}}$ from the ray $\overrightarrow{0x}$ equals $||x_{B_{n_l}}||\cdot \sin \beta$ and is $\leq ||x - x_{B_{n_l}}|| \leq K \cdot r_{n_l}+D$. Hence
	$$\sin \beta \leq  \frac{K \cdot r_{n_l}+D}{||x_{B_{n_l}}||}.$$
	We have
	\begin{align*}
		\beta &\leq \frac{\pi}{2}\sin \beta \leq \frac{\frac{\pi}{2}(K \cdot r_{n_l}+D)}{||x_{B_{n_l}}||}
	=\frac{\pi}{2} (  \frac{Kr_{n_l}}{  ||x_{B_{n_l}}||  } + \frac{D}{||x_{B_{n_l}}||} )
	\leq \frac{\pi}{2} (K\cdot \sin \alpha +\frac{D}{||x_{B_{n_l}}||} )
	\\
	&\leq \frac{\pi}{2} (K\alpha + D  \epsilon')<
	\frac{\pi}{2} (K\epsilon' + D  \epsilon').
	\end{align*}
	We get $$\alpha +\beta < \epsilon'(1+\frac{\pi}{2}(K+D))=\epsilon_1$$
	and this completes the proof.
\end{proof}

\noindent \textbf{Case 2: The strong point $x_0$ is a finite point:} Assume w.l.o.g. that $K=K(\F)$ is sufficiently large, and that $x_0$ is the origin, and let $\G=\{B_n\}_{n=1,2,\ldots} \subset \F$ be the sequence guaranteed by Claim \ref{cl:comp}. Let $\alpha$ be a small angle such that $\alpha(1+\pi K /2) < \pi /4$. Let $\L$ be a finite set
of rays emanating from the origin, such that the intersections of the cones 
$$ \C_{\ell} =\C(u_{\ell},\alpha) =   \{ x \in \Re^d: <u_{\ell},x> \geq ||x|| \cos \alpha \},$$ (where $u_{\ell}$ is a unit vector, based in the origin, on $\ell \in \L$) with  $\mathcal{S}^{d-1}$, covers  $\mathcal{S}^{d-1}$. Namely, we partition $\Re^d$ into finitely many cones around the rays in $\L$, each of which with aperture $\alpha$ (possibly some of the cones have non-empty intersection). Remove from $\G$ all sets that are pierced by the rays in $\L$, so the remainig set, which we still call $\G$, cannot be pierced by finitely many $k$-flats. (recall that here, in the induction step, $k \geq 1$.) Note that since no element in $\G$ contains the origin, $\lim_{n \to \infty} r(escribed(B_n))=0$. Since $\G$ cannot be pierced by finitely many $k$-flats, there exists a cone $\C_{\ell}$ such that the subfamily $\{ B_j \in \G : x_{B_j} \in \C_{\ell}   \}$
cannot be pierced by finitely many $k$-flats. 

Assume w.l.o.g. that $u_{\ell}=(0,\ldots,0,-1)$, namely, $\ell$ is the negative ray of the $d$-axis, and $\C_{\ell}$ is a cone with aperture $\alpha$ around $\ell$.

\begin{claim} \label{cl:wide-cone}
	For each $B_j \in \G$ with $x_{B_j} \in \C_{\ell}$, $escribed(B_j)$ is contained in a cone around $\ell$ with aperture $<\pi /4$.
\end{claim}	

\begin{proof}[Proof of Claim \ref{cl:wide-cone}]
	Consider $insc(B_j)$ which is disjoint with $\ell$, since we removed from $\G$ all the sets that are stabbed by the rays in $\L$. Let $t>r(insc(B_j))=r$ be the distance between $x_{b_j}$ and $\ell$ (see Figure \ref{fig:wide}). Let $p \in escribed(B_j)$, and by the first condition of Definition \ref{def:near-balls}, $p \in B(x_{B_j},Kr)$. W.l.o.g., $p$ is on the boundary of $B(x_{B_j},Kr)$, as illustrated in Figure \ref{fig:wide}.
	
	\begin{figure}[tb]		
		\begin{center}
			\scalebox{0.7}{
				\includegraphics{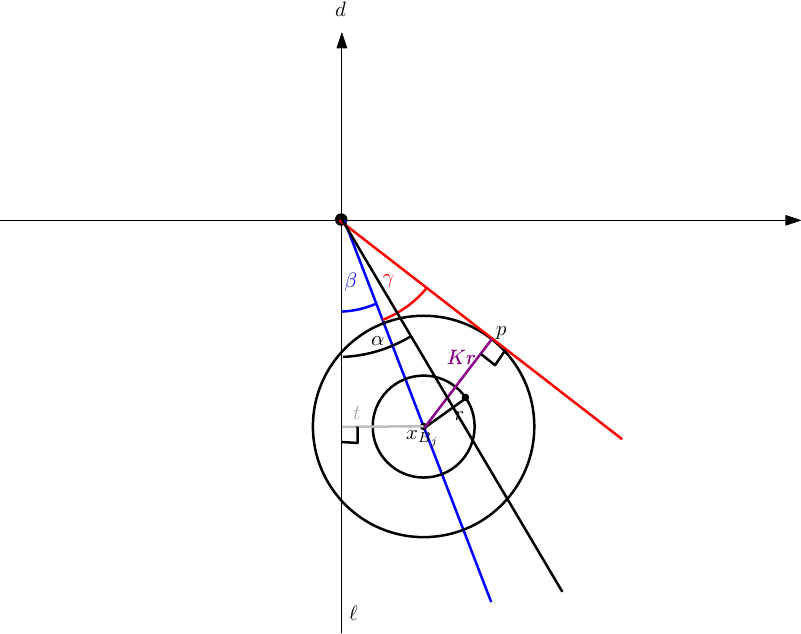}
			}
		\end{center}
		\caption{An illustration for the proof of Claim \ref{cl:wide-cone}.}
		\label{fig:wide}
	\end{figure}

	Let $\beta$ be the angle between $\ell$ and $\overrightarrow{0x_{B_j}}$, and let $\gamma$ be the angle between $\overrightarrow{0x_{B_j}}$ and $\overrightarrow{0p}$. Since $x_{B_j}\in \C_{\ell}$, we have $\beta \leq \alpha$.
	
	By basic trigonometry we have $\frac{r}{||x_{B_j}||} \leq \frac{t}{||x_{B_j}||} \leq \sin \beta \leq \sin\alpha$. Similarly, $\sin \gamma = \frac{Kr}{||x_{B_j}||}$, since the distance between $x_{B_j}$ to $\overrightarrow{0p}$ is $Kr$. By these two inequalities, $$ \sin \gamma =  \frac{Kr}{||x_{B_j}||} \leq K \sin \alpha \leq K \alpha .$$ But $\gamma \leq \frac{\pi }{2} \sin \gamma$ (since this holds for any $0 < \gamma < \pi/2$), hence $\gamma \leq \frac{\pi }{2} K \alpha$.
	
	Now, the angle between $\ell$ and $\overrightarrow{0p}$ is at most $$ \beta + \gamma \leq \alpha + \frac{\pi }{2} K \alpha = \alpha(1+\frac{\pi K }{2}) < \pi /4 ,$$ by our restriction on $\alpha$. We get that $escribed( B_{j})$ is contained in a cone with aperture $< \pi / 4$ around $\ell$.
\end{proof}

Let $\pi$ be the central projection from the origin of all elements of $\{ B_j \in \G : x_{B_j} \in \C_{\ell} \}$ onto the linear $(d-1)$-space that is tangent to $\mathcal{S}^{d-1}$ at the point $(0,\ldots,0,-1)$.
(Since $\alpha$ is small, the central projection from 0 to $\mathcal{S}^{d-1}$ and to this linear $(d-1)$-space, is almost the same.) Let the resulting set be $\G'$. Note that Claim \ref{cl:wide-cone} prevents a situation in which some $B \in \{ B_j \in \G : x_{B_j} \in \C_{\ell} \}$ surrounds the origin, or more generally, in which $\pi(B)$ is unbounded. 
  Each ball is projected by $\pi$ onto an ellipsoid with a bounded ratio between the largest and the shortest axis. Hence, the family $\G'$ consists of compact near-balls in $\Re^{d-1}$, and one can show that $K(\G')\leq \sqrt{2}K(\F)$. The exact calculation is presented in Claim \ref{cl:KtoK} below.

Since $\G'$ cannot be pierced by finitely many $(k-1)$-flats (as otherwise, the $k$-flats that each of them is spanned by one of these $(k-1)$-flats together with the origin, pierce all the sets in $\G$), by the induction hypothesis, there exists an infinite subset of $\G'$ which is $(k-1)$-independent in the linear $(d-1)$-dimentional space $\{  (x_1,\ldots,x_d): x_d=-1 \}$. Let $\bar{\G} = \{ B_{n_l}: l=1,2,\ldots  \}$ be the sequence of pre-images of this infinite subset under $\pi$. Namely, $x_{B_{n_l}} \to x_0$ and no $k$-flat $\tau$ through the origin pierces $k+1$ elements of $\bar{\G}$ (as otherwise, $\pi(\tau)$ pierces $k+1$ elemnets of $\G'$, in contradiction to the $(k-1)$-independence of $\pi(\bar{\G})$).

From now on, we construct inductively a subsequence $\{B^i\}_{i=1,2,\ldots} \subset \bar{\G}$ that is $k$-independent. As in Case 1, the first $k+1$ elements $B^1,\ldots,B^{k+1} \in \bar{\G}$ can be chosen arbitrarily. By the compactness of the elements in $\bar{\G}$, and similarly to the analysis in Case 1, the union of all $k$-flats that pierce $B^1,\ldots,B^{k+1}$ is a closed set with a positive distance from the origin. Therefore, the next set $B^{k+2} \in \bar{\G}$ can be chosen close enough to the origin, such that no $k$-flat that pierces $B^1,\ldots,B^{k+1}$ intersects $B^{k+2}$ (see Figure \ref{fig:fig5}). Continuing inductively in the same manner, (each time for all the $(k+1)$-subsets of the so-far-chosen sets), we obtain a $k$-independent subsequence $\{B^i\}_{i=1,2,\ldots}$ of $\bar{\G}$.

\begin{figure}[tb]		
	\begin{center}
		\scalebox{0.7}{
			\includegraphics{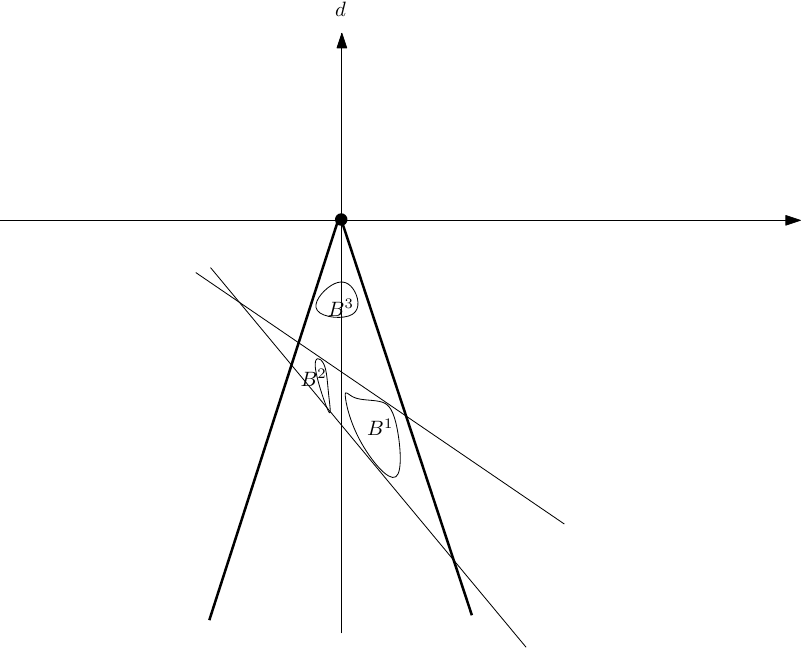}
		}
	\end{center}
	\caption{An illustration for the inductive construction in Case 2, where $k=1$ and $d=2$.}
	\label{fig:fig5}
\end{figure}

\bigskip

The only part that remains to complete the proof, is to show that $K(\G')\leq \sqrt{2}K(\F)$.

\begin{claim}\label{cl:KtoK}
	In the notations above, let $B \in \{   B_j \in \G : x_{B_j} \in \C_{\ell}  \}$. Then the ratio between $r(escribed(\pi(B)))$ and $r(insc(\pi(B)))$ is at most $\sqrt{2}K$.
\end{claim}

\begin{proof}[Proof of Claim \ref{cl:KtoK}]
	As above, we can assume that $\ell$ is the negative ray of the $d$-axis, and by Claim \ref{cl:wide-cone}, $escribed(B)$ is included in a cone around $\ell$ with aperture$<\frac{\pi}{4}$. Let $D$ be some ball centered at $x=x_B$, that is included in a $\frac{\pi}{4}$-cone around $\ell$. (At the end of the proof, we shall substitute in $D$ the two balls - $insc(B)$ and $escribed(B)$.)
	
	\begin{figure}[tb]		
		\begin{center}
			\scalebox{0.7}{
				\includegraphics{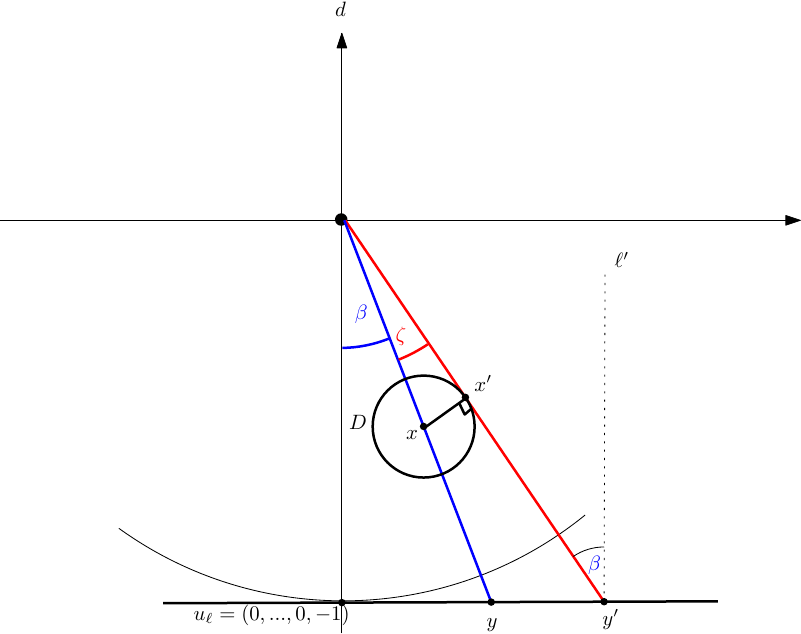}
			}
		\end{center}
		\caption{An illustration for Claim \ref{cl:KtoK}.}
		\label{fig:KtoK}
	\end{figure}

	Let $x' \in boundary(D)$ be a point that lies on a line through 0 tangent to $D$. Let $y=\pi(x)$ and $y'=\pi(x')$ (see Figure \ref{fig:KtoK}). Let $\beta$ be the angle between $\ell$ and $\overrightarrow{0x}$, and let $\zeta$ be the angle between $\overrightarrow{0x}$ and $\overrightarrow{0x'}$. In the notations of Case 2 above, $\beta \leq \alpha$ and $\beta+\zeta \leq \frac{\pi}{4}$.
	
	Note that $x$ is a fixed point, whereas $x'$ takes values within a certain $(d-2)$-subsphere of $boundary(D)$. The length which we would like to control is $||yy'||$, since the ratio between the radii $r(escribed(\pi(B)))$ and $r(insc(\pi(B)))$ is at most the maximal possible value of $||yy'||$ where $D=escribed(B)$, divided by the minimal possible value of $||yy'||$ where $D=insc(B)$.
	
	\medskip
	
	Note that the angle between $\ell$ and $\overrightarrow{0y'}$, is between $\beta+\zeta$ and $|\beta-\zeta|$ (clearly, the extremal values $|\beta \pm  \zeta|$ is obtained when $x'$ is on the plane determined by $\triangle 0u_{\ell}y$).
	Consider the ray $\ell'$ emanating from $y$ and parallel to $\ell$. The angle between $\ell'$ and $0y$ is $\beta$ (see Figure \ref{fig:KtoK}).
	Therefore, the angle $\measuredangle (0,y,y)'$ is between $(\frac{\pi}{2}-\beta)$ and $(\frac{\pi}{2}+\beta)$, and the angle $\measuredangle (0,y',y)$ is between ($\frac{\pi}{2}-\beta-\zeta)$ to $(\frac{\pi}{2}+|\beta-\zeta|$). In the right triangle $\triangle 0u_{\ell}y$, $||0y||=\frac{1}{\cos \beta}$. By the law of sines in $\triangle 0yy'$,
	$$         \frac{||yy'||}{\sin \zeta}  =  \frac{1/ \cos \beta}{\sin \measuredangle  (0,y',y)} .      $$
	Hence, the maximal possible value of $||yy'||$ is at most $\frac{\sin \zeta}{\cos \beta  \sin \frac{\pi}{4}}$ (since $\beta$+$\zeta$<$\frac{\pi}{4}$), and the minimal possible value of $||yy'||$ is at least $\frac{\sin \zeta}{\cos \beta \cdot 1}$. 
	
	When we substitute in $D$, $insc(B)$ and $escribed(B)$, the angle $\beta$ remains the same, but $\sin \zeta$ for $D=escribed(B)$ is at most $K$ times $\sin \zeta$ for $D=insc(B)$ (since $\sin \zeta$ is the ratio between $||xx'||$ and $||0x||$, and $x$ is the common center of both $insc(B)$ and $escribed(B)$). Hence, the ratio between $r(escribed(\pi(B)))$ and $r(insc(\pi(B)))$ is at most $\sqrt{2}K$, as asserted.
	  
\end{proof}

\section{Proof of Proposition~\ref{thm:ex}}\label{sec:anti}

In this section we prove Proposition~\ref{thm:ex}. Namely, we construct an infinite family of open discs in the plane that satisfies the $(3,3)$-property with respect to line transversals, but cannot be pierced by a finite number of lines.    

\begin{proof}[Proof of Proposition~\ref{thm:ex}]
	Let $\mathcal{F}=\{\F_n\}_{n=1}^{\infty} \subset \mathbb{R}^2$, where $\F_n=B(n,1/n)$ is an open disc centered at $(n,\frac{1}{n})$ with radius $\frac{1}{n}$. The family $\mathcal{F}$ does not admit a finite line transversal, since the $x$-axis meets no element of $\F$, any line that is parallel to the $x$-axis meets finitely many elements of $\F$, and any other line intersects only a finite subfamily of $\mathcal{F}$.
	
	\begin{figure}[tb]
		\begin{center}
			\scalebox{0.6}{
				\includegraphics{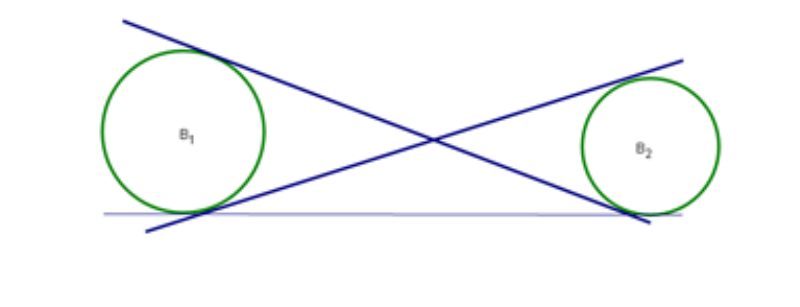}
			}
		\end{center}
		\caption{An illustration for Section \ref{sec:anti}.}
		\label{fig}
	\end{figure}
	
	On the other hand, any $\F' \subset \mathcal{F}$ which is independent w.r.t.~ lines, satisfies $|\F'| \leq 2$. Indeed, consider the two leftmost discs $B_1,B_2 \in \F'$. The right wedge that the two common inner tangents of $B_1$ and $B_2$ form, includes all elements of $\mathcal{F}$ that are to the right of $B_1$ and $B_2$ (see Figure \ref{fig}). Therefore, any element of $\mathcal{F}$ that lies to the right of $B_1$ and $B_2$ is pierced by a line that passes through $B_1$ and $B_2$, and hence cannot be contained in $\F'$.
\end{proof}

Note that the construction above admits even a stronger property than required: each finite subset of $\F$ can be pierced by a horizontal line, and each infinite subset of $\F$ cannot be pierced by finitely many lines.
We also note that no similar example could be constructed with closed balls, since by the Danzer-Gr\"{u}nbaum-Klee theorem~\cite{DGK63}, such a family would be pierced by a single line.


\section*{Acknowledgements}
The authors are grateful to Andreas Holmsen and to the anonymous reviewers for valuable suggestions and information.

\bibliographystyle{plain}
\bibliography{refs}

	
	

\end{document}